\documentclass[]{article}

\usepackage{graphicx} 
\usepackage[a4paper,left=26mm,right=26mm,top=30mm,bottom=30mm,marginpar=25mm]{geometry}
\usepackage{amsmath}
\usepackage{amsthm}
\usepackage{amssymb}

\usepackage{tikz, pgf}
\usepackage{xcolor} 
\usepackage[displaymath,mathlines]{lineno}
\usetikzlibrary{decorations.markings}
\usepackage[english]{babel}
\usepackage[T1]{fontenc}
\usepackage[utf8]{inputenc}
\usepackage{amsfonts}
\usepackage{mathrsfs}
\usepackage{xfrac}
\usepackage{lmodern}
\DeclareFontFamily{OMX}{lmex}{}
\DeclareFontShape{OMX}{lmex}{m}{n}{<-> lmex10}{}
\usepackage{fix-cm}
\usepackage{listings}
\usepackage{fancyvrb}
\usepackage{enumerate}   
\usetikzlibrary{arrows}
\usepackage{float}
\usepackage{subcaption}
\usepackage{hyperref}
\usepackage{bigints}

\renewcommand{\div}{\operatorname{div}}

\newcommand{\R}{{\mathbb{R}}}

\newcommand{\lra}{\longrightarrow}

\newcommand{\ep}{\varepsilon}

\newcommand{\cB}{{\mathcal B}}
\newcommand{\cG}{{\mathcal G}}

\newcommand{\cA}{{\mathcal A}}

\newcommand{\cL}{{\mathcal L}}

\newcommand{\cM}{{\mathcal M}}

\renewcommand{\phi}{\varphi}
\renewcommand{\a}{{\alpha}}
\renewcommand{\b}{{\beta}}

\newcommand{\s}{{\sigma}}

\newcommand{\supp}{\mathrm{supp}\;}

\usepackage{fancyhdr}

\newtheorem{theorem}{Theorem}
\numberwithin{theorem}{section}
\newtheorem{proposition}{Proposition}

\newtheorem{lemma}{Lemma}

\newtheorem{definition}{Definition}

\newtheorem{remark}{Remark}
\providecommand{\keywords}[1]{\textbf{Keywords.} #1}
\providecommand{\mathsubjclass}[2]{\textbf{2020 Mathematics Subject Classification.} #1}

\author{Fabio Bagagiolo\thanks{University of Trento, Department of Mathematics, Via Sommarive 14, 38123 Povo (TN), {\tt fabio.bagagiolo@unitn.it}.\\
This author was partially supported by MUR-PRIN2020 Project (No. 2020JLWP23) “Integrated mathematical approaches to socio-epidemiological dynamics” and by ``INdAM-GNAMPA Project'' (CUP E53C23001670001) ``Modelli mean-field games per lo studio dell'inquinamento atmosferico e i suoi impatti''.}, Rossana Capuani\thanks{University of Arizona, Department of Mathematics, 617 N. Santa Rita Ave., Tucson, USA, {\tt rossanacapuani@arizona.edu}.\\
This author was partially supported by ``INdAM-GNAMPA Project'' (CUP E53C23001670001) ``Controllo ottimo infinito dimensionale: aspetti teorici ed applicazioni''.}, Luciano Marzufero\thanks{Faculty of Economics and Management, Free University of Bozen-Bolzano, Piazza Universit\`a 1, 39100 Bolzano (BZ), {\tt luciano.marzufero@unibz.it}.\\
This author was partially supported by ``INdAM-GNAMPA Project'' (CUP E53C23001670001) ``Modelli mean-field games per lo studio dell'inquinamento atmosferico e i suoi impatti'' and by the Gruppo Nazionale per l'Analisi Matematica e le loro Applicazioni (GNAMPA-INdAM).}}

\date{}

\begin{document}

\title{\Large \bf A zero-sum differential game for two opponent masses}
\maketitle
\begin{abstract}
We investigate an infinite dimensional partial differential equation of Isaacs' type, which arises from a zero-sum differential game between two masses. The evolution of the two masses is described by a controlled transport/continuity equation, where the control is given by the velocity vector field. Our study is set in the framework of the viscosity solutions theory in Hilbert spaces, and we prove the uniqueness of the value functions as solutions of the Isaacs equation. 
\end{abstract}
\par\smallskip\noindent
\mathsubjclass{Primary: 49N70; Secondary: 49N75; 49L12; 49L25; 35Q49.}
\par\bigskip\bigskip\bigskip\noindent
\section{Introduction}
In the recent paper \cite{BCM}, the present authors study a finite-horizon differential game of pursuit-evasion type between a single player and a mass of agents. The player and the mass directly control their own evolution, which for the single agent is an ordinary differential equation and for mass is given by a first order partial differential equation of transport/continuity type.  By dynamic programming techniques an infinite dimensional Isaacs equation is derived and it is proved that the value function is the unique viscosity solution on a suitable invariant subset of a Hilbert space. 
\par
In the present paper, we extend the model to a competition between two opponent masses/population of agents, say: population $X$ and population $Y$. In this case, the controlled dynamics are both given by the following transport/continuity equation
\begin{equation}
\label{eq:intro_continuity}
m^i_t+\div(\beta^im^i)=0,
\end{equation}
where, for $i\in\{X,Y\}$, $m^i:\mathbb{R}^d\times[0,T]\lra[0,+\infty[$ is the time-dependent mass-distribution of the population $i$ (i.e., for $t\in[0,T]$, $m^i(t):=m^i(\cdot ,t):\mathbb{R}^d\lra[0,+\infty[$ is the spatial distribution at time $t$) and $\beta^i:\mathbb{R}^d\times[0,T]\lra\mathbb{R}^d$ is the time-dependent control of the population $i$, that regulates the velocity and the spreading of the mass in $\mathbb{R}^d$ (i.e., for $t\in[0,T]$, $\beta^i(t):=\beta^i(\cdot,t):\mathbb{R}^d\lra\mathbb{R}^d$ is the space-control activated at time $t$).
\par
The two populations are competing in a pursuer-evader fashion, and we model such a competition by a cost as
$$
J(\bar m^X,\bar m^Y,t,\beta^X,\beta^Y)=\\
\int_t^T\ell(m^X(s),m^Y(s),s,\beta^X(s),\beta^Y(s))ds+g(m^X(T),m^Y(T)),
$$
where $\ell$ and $g$ are two suitable functions, $s\longmapsto m^X(s),m^Y(s)$ are the evolutions, as space-dependent functions on $\mathbb{R}^d$, given by the corresponding equations (\ref{eq:intro_continuity}), respectively governed by the controls $s\longmapsto \beta^X(s),\beta^Y(s)$ and starting, at time $t$, from $\bar m^X,\bar m^Y:\mathbb{R}^d\lra[0,+\infty[$. The competition is then analytically realized by the fact that the population $X$ wants to minimize $J$ and the population $Y$ wants to maximize.
\par
Using an adapted concept of non-anticipating strategies $\gamma^X,\gamma^Y$ (see \cite{ellkal} and \cite{BCD}), we define the lower and the upper values of the game as, respectively,
\begin{align*}
\underbar V(\bar m^X,\bar m^Y,t)&=\inf_{\gamma^X}\sup_{\beta^Y}J(\bar m^X,\bar m^Y,t,\gamma^X[\beta^Y],\beta^Y),\\
\overline{\text V}(\bar m^X,\bar m^Y,t)&=\sup_{\gamma^Y}\inf_{\beta^X}J(\bar m^X,\bar m^Y,t,\beta^X,\gamma^Y[\beta^X]),
\end{align*}
where the non-anticipating strategy $\gamma^i$ is a suitable function taking controls for the opponent population $\hat i\in\{X,Y\}\setminus\{i\}$ and giving a control for the population $i$ in such a way that it is not dependent upon future behavior.
\par
Since the controlled evolutions $s\mapsto m^i(s)$ given by the equations in (\ref{eq:intro_continuity}) are evolutions in a functions space (possibly a Hilbert space), by dynamic programming techniques we get that both lower and upper values are viscosity solutions of a final condition Cauchy problem for an infinite dimensional Hamilton-Jacobi-Isaacs equation which is, respectively
\begin{equation}
\label{eq:intro_isaacs1}
\begin{cases}
-\underbar V_t+\\
\inf_{b^Y}\sup_{b^X}\left\{\langle D_X\underbar V,{\rm div}(b^X\bar m^X)\rangle+\langle D_Y\underbar V,{\rm div}(b^Y\bar m^Y)\rangle
-\ell(\bar m^X,\bar m^Y,t,b^X,b^Y)\right\}=0,\\
\underbar V(\bar m^X,\bar m^Y,T)=g(\bar m^X,\bar m^Y),
\end{cases}
\end{equation}
\begin{equation}
\label{eq:intro_isaacs2}
\begin{cases}
-\overline{\text V}_t+\\
\sup_{b^X}\inf_{b^Y}\left\{\langle D_X\overline{\text V},{\rm div}(b^X\bar m^Y)\rangle+\langle D_Y\overline{\text V},{\rm div}(b^Y\bar m^Y)\rangle
-\ell(\bar m^X,\bar m^Y,t,b^X,b^Y)\right\}=0,\\
\overline{\text V}(\bar m^X,\bar m^Y,T)=g(\bar m^X,\bar m^Y),
\end{cases}
\end{equation}
where $\langle\cdot,\cdot\rangle$ is the scalar product in the Hilbert space, $D_i$ is the Fr\'echet differential with respect to $m^i$, and $b^i$ is in $B^i$, the set of the admissible constant controls $b^i:\mathbb{R}^d\lra\mathbb{R}^d$ for the population $i$. The first problem \eqref{eq:intro_isaacs1} is called the {\it upper problem} and the second one \eqref{eq:intro_isaacs2} the {\it lower problem}.
\par
By suitably restricting the domain $\cM\ni (\bar m^X,\bar m^Y,t)$ of the equations in \eqref{eq:intro_isaacs1}-\eqref{eq:intro_isaacs2}, however coherently with the controlled evolution in (\ref{eq:intro_continuity}), we get a uniqueness result proving that the lower value $\underbar V$ (resp. the upper value $\overline{\text V}$) is the unique continuous viscosity solution of the upper problem (resp. the lower problem). In particular, when the two Isaacs equations coincide (for example when the running cost $\ell$ does not depend on the controls $b^X,b^Y$), then, by uniqueness, $\underbar V=\overline{\text V}$ which means that the game has a value. 
\par
Similarly to \cite{BCM}, the main issues in performing what we said above, are that the differential game is of finite horizon type and, more crucially, that the Isaacs equations are infinite dimensional. This entails that, especially for uniqueness but not only, the restricted domain $\cM$ as above must be suitably constructed, in particular for overcoming the lacking of compactness in the infinite dimensional spaces. 
\par
In \cite{BCM} one of the two opponents is a single (finite dimensional) player and the motivation for that was a possible further study of a mean field zero-sum game between two populations. In the present paper, the opponents are two (infinite dimensional) players, the two masses. The study of a game between two opponent masses seems to be rather new, in particular in connection with minimax equilibria and the corresponding Isaacs equation set in a Hilbert space. Some problems of this type, in a rather weak set as the Wasserstein probability space, are studied in \cite{marquin, gangbo, jimmarquin, moonbasar, marigondacapuani}. However, here the goal is to work in a Hilbert setting, that is in the case where the measures on $\mathbb{R}^d$, representing the distribution of the masses, have a density with respect to the Lebesgue measure. We believe that this approach may be more prone to possible real applications, in particular concerning the construction of an optimal feedback. Some of those possible applications may be found, for example, in \cite{pattur, gomes1, colombogaravello, kolokoltsov1, kolokoltsov2, gomes2, wang}.
\par
The study of viscosity solutions for infinite dimensional Hamilton-Jacobi equations goes back to the early stage of the viscosity solutions theory \cite{crandall1, crandall2} (see also \cite{BCD} for the finite dimensional case). We also quote \cite{kocsorswi} for studies of an infinite dimensional Isaacs equation which however presents many different features with respect to the present one. In particular, in the present model, the controls directly act on the differential operator in the evolution equations (\ref{eq:intro_continuity}), which are first order equations. Other games with more than one masses are of course already studied, in particular we remind the reader to the papers \cite{achdou, cirantverzini, CFM, pedro} where, however, the game is not of the minimax form, as it is in the present paper.
\par
As already said, in this paper we set the problem in a Hilbert space, and we suitably settle it in a viable compact subsets of the space. This approach seems also to be useful for the study of a further model where the first order continuity equations (\ref{eq:intro_continuity}) are replaced by second order parabolic equations of the form 
$$
m^i_t-\sigma\Delta m^i+{\rm div}(\beta^im^i)=0.
$$
In Section \ref{sec:parabolic} we argue on such a possible future study and we also point out an interesting possible ``vanishing viscosity'' feature as $\sigma\to0^+$.
\par
In Section \ref{sec:examples} we report direct calculations on some one-dimensional examples ($d=1$), and we also argue on the extreme difficulty in doing that, due to the infinite dimensional feature of the problem.
\par
The plan of the rest of the paper is the following: in Section \ref{continuityeq1} we give some results on the solution of \eqref{eq:intro_continuity}; in Section \ref{dynprogreg} we introduce the differential game and its notation. Moreover we prove the dynamic programming principle and study the corresponding Isaacs equation.
\section{On the mass distribution: estimates and technicalities}
\label{continuityeq1}
In this section, we provide some preliminary results and useful remarks on the evolution of a mass on $\R^d$, which will be necessary for the next sections. We refer to \cite{BCM} for more details, proofs and general references.
\subsection{The continuity equation for the mass}
\label{continuityeq2}
The mass distribution on $\R^d$ can be represented by the evolution $t\longmapsto\mu(t)$ of a measure on $\R^d$, where the quantity of mass in $A\subset\R^d$ at time $t$ is the measure of $A$ at the same time: $\mu(t)(A)$. When the mass is moving according to a given time-dependent vector field $\beta:\R^d\times[0, T]\lra\R^d$, then, at least formally, the evolution of the measure $\mu$ satisfies the following continuity equation (in the distributional sense for measures):
\begin{equation}
\label{sec1:continuity}
\begin{cases}
\mu_s(x, s) + \div(\beta(x,s)\mu(x, s))=0,&(x, s)\in\mathbb{R}^d\times]t, T[\\
\mu(\cdot, t)=\bar \mu,
\end{cases}
\end{equation}
where $\bar\mu$ is the value of the measure at the initial time $t$, $T>0$ is the final time and, here and in the sequel, $\div$ stands for the spatial divergence.
\par
The well-posedness of \eqref{sec1:continuity} is closely related to the well-posedness of the following Cauchy problem in $\mathbb{R}^d$.
\begin{equation}
\label{ode}
\begin{cases}
y'(s)=\beta(y(s), s),&s\in]t, T[\\
y(t)=x.
\end{cases}
\end{equation}
Here we assume the following hypothesis: 
\begin{itemize}
\item[(H)]
the time-dependent controls $\beta$ in \eqref{ode} and \eqref{sec1:continuity} are functions from $[0, T]$ to the set $\tilde\cB$, which, for a given fixed $M>0$, is defined as
\begin{multline}
\label{campispaziali}
\tilde\cB:=\left\{b\in W^{2, \infty}(\R^d, \R^d)\cap H^1(\R^d, \R^d):\right.\\
\left.\|b\|_{L^{\infty}(\R^d, \R^d)}\leq M,\ \|b\|_{H^1(\R^d, \R^d)}\leq M,\ \|\div b\|_{W^{1, \infty}(\R^d)}\leq M\right\},
\end{multline}
and moreover they belong to 
\begin{equation}
\label{vectorfields}
\cB:=\{\beta\in L^2([0, T], W^{2, \infty}(\R^d, \R^d)\cap H^1(\R^d, \R^d)):\beta(\cdot, t)\in\tilde\cB\ \text{ for a.e. }t\}.
\end{equation}
\end{itemize}
Under such hypothesis, the existence and the uniqueness of solutions to equations \eqref{ode} and \eqref{sec1:continuity}  are guaranteed for any initial data $(x, t)$ and Borel measure $\bar\mu$, respectively. Weaker conditions may also suffice, as noted in references \cite{DiPernaLions, CrippaTesi, AmbCri2014}. In particular, the solution to \eqref{sec1:continuity} is given by the push-forward of $\bar\mu$ through the flow $\Phi(\cdot,t,s):\mathbb{R}^d\lra\mathbb{R}^d$ associated to \eqref{ode}. Due to the (only spatial) regularity assumed in \eqref{campispaziali} for the admissible controls, we have the following first remark.
\begin{remark}
\label{remark1}
If all the fields $b\in\tilde{\cal B}$ have their support contained in the same bounded subset of $\R^d$, then the $L^2$-closure $\overline{\tilde\cB}$ of \eqref{campispaziali} is a compact subset of $L^2(\mathbb{R}^d,\mathbb{R}^d)$. This is important because the compactness of the admissible constant in time controls (which here are functions of the space $x$, anyway) will be crucial in order to construct suitable non-anticipating strategies for deriving the infinite dimensional Isaacs equations satisfied by the values of the game.
\end{remark}
When $\bar\mu$ is absolutely continuous with respect to the Lebesgue measure $\cL^d$, that is $\bar\mu= \bar m\mathcal{L}^d$, where $\bar m:\R^d\lra\R$ is the density, all the measures $\mu(\cdot, s)$ are absolutely continuous with respect to $\mathcal{L}^d$. Their density $m(\cdot, s)$ can be explicitly computed as (see \cite{AmbCri2014})
\begin{equation}
\label{eq:formula}
m(\cdot, s)=\frac{\bar m(\cdot)}{{|\rm{det}} J \Phi(\cdot, t, s)|} \circ \Phi^{-1}(\cdot, t, s),
\end{equation}
where $\Phi^{-1}(\cdot, t, s)$ is the inverse of the flow associated to \eqref{ode} and $J\Phi$ is its Jacobian matrix. In the following, let $m(x, s; \beta, t, \bar m)$ denote the value at $(x, s)\in\R^d\times]t, T]$ of the density of the solution of \eqref{sec1:continuity} with $\beta\in\cB$, where the initial datum at time $t$, $\bar\mu$, has density $\bar m$. From formula \eqref{eq:formula} and the regularities in \eqref{campispaziali}-\eqref{vectorfields}, the following results can be proved.
\begin{proposition}
\label{propositionestimates}
Suppose that $(H)$ holds and $\bar m\in H^1(\mathbb R^d)\cap W^{1, \infty}(\R^d)$. Then
\begin{itemize}
\item[(i)] for every $t\in[0, T]$ and $s\in[t, T]$ 
$$
x\longmapsto m(x, s; \beta, t, \bar m)\text{ belongs to }H^1(\mathbb R^d)\cap W^{1, \infty}(\R^d),
$$
\item[(ii)] for every $t\in[0, T]$
$$
m(\cdot, \cdot; \beta, t, \bar m)\in C^0([t, T], H^1(\R^d)\cap W^{1, \infty}(\R^d)),
$$
where the modulus of continuity is independent of $t, \beta$ and only depends on the norms of $\bar m$;
\item[(iii)] for every $t\in[0, T]$ and $s\in[t, T]$ 
\begin{equation}\label{stimamw}
\|m(\cdot, s; \beta, t, \bar m)\|_{W^{1, \infty}(\R^d)}\leq e^{MT}\sup\{1+\tilde M,L_{\Phi^{-1}}\}\|\bar m(\cdot)\|_{W^{1, \infty}(\R^d)},
\end{equation}
where the constant $\tilde M>0$ encloses the bounds for the spatial derivative of $\div\beta$, $\Phi$ and $\Phi^{-1}$, independent of $t, \beta$, and $L_{\Phi^{-1}}$ is the Lipschitz constant of $\Phi^{-1}(\cdot, t, s)$.
\end{itemize}
Moreover, it holds the following
\begin{equation}
\label{dipendenzacont1}
\|m(\cdot, s; \beta, t_1, \bar m^1)-m(\cdot, s; \beta, t_2, \bar m^2)\|_{H^1(\mathbb R^d)}\leq\tilde L\left(|t_1-t_2|+\|\bar m^1(\cdot)-\bar m^2(\cdot)\|_{H^1(\mathbb R^d)}\right),
\end{equation}
for every $t_1,t_2\in[0, T]$, $s\in[\max\{t_1,t_2\}, T]$ and for some $\tilde L$ independent of $t_1,t_2$ and $\beta$.
\end{proposition}
Thanks to Proposition \ref{propositionestimates}, we have the following second remark.
\begin{remark}
\label{osscontinuita}
By Proposition \ref{propositionestimates} and assumption $(H)$, we observe that $\div(\beta m), \frac{\partial m}{\partial t}\in L^2(\R^d\times[0, T])$. Moreover, since $\mu$ is the solution of \eqref{sec1:continuity} in the sense of measures, we have that its density $m$ is solution, in the distributional sense for measures, of $\frac{\partial m}{\partial t}+\div(\beta m)=0$. Since $\frac{\partial m}{\partial t}, \div(\beta m)\in L^2(\R^d\times[0, T])$, $m$ is the unique solution in $L^2(\R^d\times[0, T])$ (see \cite{CrippaTesi}). In particular, we have that $\frac{\partial m}{\partial t}=-\div(\beta m)$ as functions in $L^2(\R^d\times[0, T])$. From this, we can conclude that if $\phi\in C^1(L^2(\R^d))$, then the map
\begin{align*}
\xi&:[t, T]\lra\R\\
\tau&\longmapsto\xi(\tau)=\phi(m(\cdot, \tau; \beta, t, \bar m))
\end{align*}
is differentiable and $\xi'(\tau)=-\langle D_m\phi, \div(\beta m)\rangle_{L^2(\R^d)}$, where $D_m\phi$ denotes the Fr\'echet derivative of $\phi$ w.r.t. $m$ and $\langle\cdot, \cdot\rangle_{L^2(\R^d)}$ the scalar product in $L^2(\R^d)$. This last fact is essential for working with dynamic programming and viscosity solutions.
\end{remark}
\subsection{A suitable invariant set for the evolution of the mass}
\label{sec:invariantset}
The aim of this section is to construct a suitable set $\cM$ which is compact in $L^2(\R^d)\times[0, T]$ and invariant for the controlled evolution $m(\cdot, s, \a, t, \bar m)$. Up to proper adjustments, this will be the set in which we are going to consider the Hamilton-Jacobi-Isaacs equation of the differential game (see \S\ref{HJIforV}).
\par
We fix an open bounded subset $\Omega\subset\R^d$ and a constant $K>0$ such that all the admissible initial distributions of the mass belong to
$$
\cG=\left\{m_0\in W^{1,\infty}(\Omega):\|m_0\|_{W^{1,\infty}(\Omega)}\leq K,\ \supp m_0\subset\Omega\right\}.
$$
Referring to point $(iii)$ in the Proposition \ref{propositionestimates}, we define $B:=\max\{1+\tilde M,L_{\Phi^{-1}}\}K$. For all $m_0\in\cG$, $\b\in\cB$ and $s\in[0,T]$, we obtain
\begin{equation}
\label{stimanuovanuova}
\|m(\cdot,s;\b, 0, m_0)\|_{W^{1,\infty}(\Omega)}\le Be^{MT}.
\end{equation}
Recalling that $M$ is the bound for $\|\b(\cdot,t)\|_\infty$ (see \eqref{campispaziali}), we define, for all $t\in[0,T]$ and for all $s\in[t,T]$, the bounded sets
$$
\Omega_1(t)=\overline B(\Omega, Mt),\ \ \Omega_1(s,t)=\overline B\left(\Omega_1(t),M(s-t)\right),
$$
being, as example, $\overline{B}(\Omega, Mt)=\{x\in\R^d:\text{dist}(x, \Omega)\leq Mt\}$. A sort of semigroup property holds: $\Omega_1(0)=\overline\Omega$, $\Omega_1(s,t)=\Omega_1(s)$ for all $t\in[0,s]$ and, in particular, $\Omega_1(T)=\Omega_1(T,t)$. Namely, assuming $\bar m\in W^{1, \infty}(\R^d)\cap H^1(\R^d)$ with $\supp\bar m\subset\Omega_1(t):=\overline{B}(\Omega, Mt)$, we have, independently of $\b\in\cB$,
$$
\supp m(\cdot, s'; \b, \bar m^X)\subset\Omega_1(s',t)\subset\Omega_1(s'',t)\subset\Omega_1(T,t)=\Omega_1(T)\label{stimanuova0}\\
$$
for all $t\le s'\le s''\le T$. Observe that all such functions $m$ can be extended to the entire set $\Omega_1(T)$ by setting them equal to $0$ outside their support. Now, for every $t\in[0, T]$, let $K(t)\subset H^1(\Omega_1(T))$ be a compact set such that $K(t')\subset K(t)$ for every $0\leq t'\leq t\leq T$, $m\in K(t)$ if $K(t_n)\ni m_n\lra m$ in $H^1(\Omega_1(T))$ and $t_n\to t$ in $[0, T]$, and for every $0\leq t'\leq t\leq T$, $K(t)$ contains $m(\cdot, t; \beta, t', \bar m)$ for every $\beta\in\cB$ and $\bar m\in K(t')$. Observe that, by \eqref{dipendenzacont1} and suitable possible convergences on the set of controls $\beta$, a simple example of such compact sets $K(t)$ may be the points of all trajectories up to the time $t$ started at time $0$ from points of a compact set $K$ in $H^1(\Omega_1(T))$. Recalling \eqref{stimanuovanuova}, then for all $t\in[0,T]$, we define the following sets:
\begin{multline*}
\cM(t):=\left\{\bar m\in W^{1, \infty}(\Omega_1(T))\cap K(t):\supp\bar m\subset\Omega_1(t),\right.\\
\left. \|m(\cdot, s; \b, t, \bar m)\|_{W^{1, \infty}(\Omega_1(T))}\leq Be^{MT} \ \text{for all }s\in[t,T]\text{ and }\b\in\cB\right\};
\end{multline*}
\begin{multline*}
\cM=\bigcup_{t\in[0, T]}\left(\cM(t)\times\{t\}\right)=\left\{(\bar m, t)\in W^{1, \infty}(\Omega_1(T))\cap K(t)\times[0, T]:\right.\\
\left.\supp\bar m\subset\Omega_1(t),\ \|m(\cdot, s; \b, t, \bar m)\|_{W^{1, \infty}(\Omega_1(T))}\leq Be^{MT}\ \text{for all }s\in[t,T] \text{ and }\ \beta\in\cB\right\}.
\end{multline*}
By construction, $\cM$ is compact in $H^1(\Omega_1(T))\times[0, T]$ and hence also in $L^2(\Omega_1(T))\times[0, T]$, as indeed similarly requested in \cite{BCM}.
\par
Thanks to compactness, the $H^1$ and $L^2$ norms are topologically equivalent in $\cM$. That is, there exists a modulus of continuity $\tilde\omega$ such that for every $(m_1 ,t_1),(m_2, t_2)\in\cM$ one has 
\begin{equation}
\label{modcontm}
\|m_1-m_2\|_{H^1(\Omega_1(T))}\leq\tilde\omega\left(\|m_1-m_2\|_{L^2(\Omega_1(T))}\right).
\end{equation}
Furthermore, $\cM$ contains all the possible evolutions $s\longmapsto(m(\cdot, s; \b, t, \bar m), s)=(m(s), s)$ (with field $\beta\in\cB$) starting from any $(\bar m, t)\in\cM$. In fact, for all $s\in[t,T]$ and for $\tau\in[s,T]$, it holds
$$
\|m(\cdot,\tau;\b,s,m(s))\|_{W^{1, \infty}(\Omega_1(T))}=\|m(\cdot,\tau;\b,t,\bar m)\|_{W^{1, \infty}(\Omega_1(T))}\le Be^{MT}.
$$
Thus $(m(s), s)\in\cM$ because $\supp m(s)\subset\Omega_1(s)$ since $\supp\bar m\subset\Omega_1(t)$. Therefore $\cM$ is invariant for the evolution of the continuity equation with field $\b\in\cB$. Specifically, it contains all the trajectories $t\longmapsto(m(t), t)$ starting from $(m_0, 0)$ with $m_0\in\cG\cap K(0)$.
\section{The differential game model}
\label{dynprogreg}
In this section, we introduce the model of the differential game between two opponent masses, we prove the dynamic programming principle and study the corresponding Isaacs equation.
\par
The controlled equations for the evolution of the masses $m^X$ and $m^Y$ are given by
\begin{equation}
\label{contequationformassX}
\begin{cases}
m^X_s(\cdot, s)+\div(\a(\cdot, s)m^X(\cdot, s))=0,&s\in]t, T]\\
m^X(\cdot, t)=\bar m^X,
\end{cases}
\end{equation}
\begin{equation}
\label{contequationformassY}
\begin{cases}
m^Y_s(\cdot, s)+\div(\b(\cdot, s)m^Y(\cdot, s))=0,&s\in]t, T]\\
m^Y(\cdot, t)=\bar m^Y,
\end{cases}
\end{equation}
where $t\in[0, T]$, $\bar m^i\in H^1(\R^d)\cap W^{1, \infty}(\R^d)$, $i\in\{X,Y\}$, and the controls $\a$ and $\b$ belong to
\begin{equation}
\label{controlsX}
\cA(t):=\{\a\in L^2([t, T], W^{2, \infty}(\R^d, \R^d)\cap H^1(\R^d, \R^d)):\a(\cdot, s)\in\tilde\cA\ \text{ for a.e. }s\}
\end{equation}
and
\begin{equation}
\label{controlsY}
\cB(t):=\{\b\in L^2([t, T], W^{2, \infty}(\R^d, \R^d)\cap H^1(\R^d, \R^d)):\b(\cdot, s)\in\tilde\cB\ \text{ for a.e. }s\}
\end{equation}
respectively, where the sets $\tilde\cA$ and $\tilde\cB$ are taken as \eqref{campispaziali} with the constant $M$ which can be generally different for each population.
\par
Consequently, the same holds for the estimates in Proposition \ref{propositionestimates}: the constants in \eqref{stimamw}-\eqref{dipendenzacont1} can generally be different for populations $X$ and $Y$. Note also that, in contrast to the previous section and the Introduction, in \eqref{controlsX}-\eqref{controlsY} we explicitly introduce the initial-time-dependence of the controls in the notation. We call $\alpha$ the controls for the mass $X$ (named as $\beta^X$ in the Introduction) and $\beta$ the controls for the mass $Y$ (named as $\beta^Y$ in the introduction). Coherently, in the following we will use $a$ to denote a generic constant (in time) control for the mass $X$ belonging to $\tilde\cA$, as well as $b$ to denote for a generic constant (in time) control for the mass $Y$ belonging to $\tilde\cB$.
\par
Let us consider the following.
\begin{itemize}
\item Running cost
\begin{align*}
\ell:H^1(\R^d)\times H^1(\R^d)\times[0, T]\times L^2(\R^d)\times L^2(\R^d)&\lra[0, +\infty[\\
(\bar m^X, \bar m^Y, s, a, b)&\longmapsto\ell(\bar m^X, \bar m^Y, s, a, b),
\end{align*}
and suppose that it is bounded, strongly continuous and uniformly strongly continuous w.r.t. $(\bar m^X, \bar m^Y, t)$ uniformly w.r.t. $(a, b)$, that is, there exists a modulus of continuity $\omega_{\ell}$ such that, for any fixed $a, b$ it is
\begin{multline*}
|\ell(\bar m_1^X, \bar m_1^Y, s_1, a, b)-\ell(\bar m_2^X,\bar m_2^Y, s_2, a, b)|\\
\leq\omega_{\ell}\left(\|\bar m_1^X-\bar m_1^Y\|_{H^1(\R^d)}+\|\bar m_2^X-\bar m_2^Y\|_{H^1(\R^d)}+|s_1-s_2|\right)
\end{multline*}
for all $(\bar m_1^X, \bar m_1^Y, s_1), (\bar m_2^X, \bar m_2^Y, s_2)\in H^1(\R^d)\times H^1(\R^d)\times[0, T]$.
\item Final cost
\begin{align*}
\psi:H^1(\R^d)\times H^1(\R^d)&\lra[0, +\infty[\\
(\bar m^X, \bar m^Y)&\longmapsto\psi(\bar m^X, \bar m^Y),
\end{align*}
and suppose that it is bounded and uniformly strongly continuous w.r.t. $(\bar m^X, \bar m^Y)$. We denote by $\omega_\psi$ its modulus of continuity.
\end{itemize}
The corresponding cost functional $J$ is given by
\begin{multline*}
J(\bar m^X, \bar m^Y, t, \a, \beta)
=\int_t^T\ell\left(m^X\left(\cdot, s;\a, t, \bar m^X\right), m^Y\left(\cdot, s; \beta, t, \bar m^Y\right), s, \a(\cdot, s), \beta(\cdot, s)\right)ds\\
+\psi\left(m^X\left(\cdot, T; \a, t, \bar m^X\right), m^Y\left(\cdot, T; \beta, t, \bar m^Y\right)\right),
\end{multline*}
for all $(\bar m^X, \bar m^Y, t, \a, \b)\in H^1(\R^d)\times H^1(\R^d)\times[0, T]\times\cA(t)\times\cB(t)$.
\par\smallskip
In our model, the mass $m^X$ aims to minimize the cost $J$ while the mass $m^Y$ aims to maximize it. Therefore, let us first introduce the non-anticipating strategies for the mass $m^X$
$$
\Gamma(t)=\{\gamma:\cB(t)\lra\cA(t): \beta\longmapsto\gamma[\beta]\ \text{non-anticipating}\}.
$$
We point out that ``non-anticipating'', in this context, means that for all $\tau\in[t, T]$,
$$
\beta_1(\cdot, s)=\beta_2(\cdot, s)\ \text{in $\tilde\cB$ a.e. $s$}\in [t, \tau]\ \Rightarrow \gamma[\beta_1](\cdot, s)=\gamma[\beta_2](\cdot, s)\ \text{in $\tilde\cA$}\ \text{a.e.}\ s\in[t, \tau].
$$
Next, we define the lower value function associated to the cost functional $J$ as follows:
\begin{equation}
\label{lowervaluefun}
\underbar{V}(\bar m^X, \bar m^Y, t)=\inf_{\gamma\in\Gamma(t)}\sup_{\beta\in\cB(t)}J(\bar m^X, \bar m^Y, t, \gamma[\beta], \beta).
\end{equation}
In the same way, we can define the non-anticipating strategies for the mass $m^Y$:
$$
\Delta(t)=\{\delta:\cA(t)\lra\cB(t): \alpha\longmapsto\delta[\alpha]\ \text{non-anticipating}\},
$$
and consider the upper value function
$$
\overline{\text V}(\bar m^X, \bar m^Y, t)=\sup_{\delta\in\Delta(t)}\inf_{\alpha\in\cA(t)}J(\bar m^X, \bar m^Y, t, \alpha, \delta[\alpha]).
$$
As mentioned in the Introduction, $\underbar{V}$ and $\overline{\text V}$ will be the unique continuous viscosity solutions of the upper Isaacs problem and of the lower Isaacs problem, respectively. In the following, we will focus on $\underbar V$ and the upper problem, as the lower problem is similar and symmetric.

\subsection{Dynamic Programming Principle and regularity of the value function}
In this section, we deduce first the dynamic programming principle and then a regularity result for the value function.
\begin{proposition}
\label{dynprog}
Suppose that the assumptions on $\ell$, $\psi$ and \eqref{vectorfields} hold. Then the lower value function $\underbar V$, defined in \eqref{lowervaluefun}, satisfies the Dynamic Programming Principle. More precisely, for all $(\bar m^X, \bar m^Y, t)\in H^1(\mathbb R^d)\times H^1(\mathbb R^d)\times[0, T]$ and for all $\tau\in[t, T]$,
\begin{multline}
\label{dpp}
\underbar V(\bar m^X, \bar m^Y, t)\\
=\inf_{\gamma\in\Gamma(t)}\sup_{\beta\in\cB(t)}\Bigg(\int_t^{\tau}\ell\left(m^X\left(\cdot, s; \gamma[\beta], t, \bar m^X\right), m^Y\left(\cdot, s; \beta, t, \bar m^Y\right), s, \gamma[\beta](\cdot, s), \beta(\cdot, s)\right)ds\\+\underbar V(m^X\left(\cdot, \tau; \gamma[\beta], t, \bar m^X\right), m^Y(\cdot, \tau; \beta, t, \bar m^Y), \tau)\Bigg).
\end{multline}
\end{proposition}
\begin{proof}
For all $(p, q, \tau)\in H^1(\mathbb R^d)\times H^1(\mathbb R^d)\times[0, T]$ and for all $\varepsilon>0$ we choose $\gamma_{(p, q, \tau)}\in\Gamma(\tau)$ such that
$$
\underbar V(p, q, \tau)\geq \sup_{\beta\in \cB(\tau)}J(p, q, \tau, \gamma_{(p, q, \tau)}[\beta], \beta)-\varepsilon.
$$
Let $t\in[0,\tau]$. For simplicity, we denote by $w(\bar m^X, \bar m^Y, t)$ the right-hand side of \eqref{dpp}. In order to prove the \eqref{dpp},  we show the following:
\begin{itemize}
\item[$(i)$] $\underbar V(\bar m^X, \bar m^Y, t)\leq w(\bar m^X, \bar m^Y, t)$;
\item[$(ii)$] $\underbar V(\bar m^X, \bar m^Y, t)\geq w(\bar m^X, \bar m^Y, t)$.
\end{itemize}
As usual, the proof of $(ii)$ is more involved than the one of $(i)$ (see for example \cite{BCD} for the finite dimensional case). Here, we prove $(i)$ for the present case with two masses. For $(ii)$, we refer to \cite{BCM}, with appropriate modifications for the two masses scenario.
\par\smallskip
Let us consider $\bar\gamma\in\Gamma(t)$  such that
\begin{multline*}
w(\bar m^X, \bar m^Y, t)\geq\sup_{\beta\in\cB(t)}\Bigg(\int_t^{\tau}\ell(m^X(\cdot, s;\bar\gamma[ \beta], t, \bar m^X), m^Y(\cdot, s; \beta, t, \bar m^Y), s, \bar\gamma[\beta](s), \beta(\cdot, s))ds\\+\underbar V(m^X(\cdot, s; \bar\gamma[\beta], \tau, \bar m), m^Y(\cdot, \tau; \beta, t, \bar m), \tau)\Bigg)-\varepsilon.
\end{multline*}
Taken $\beta\in\cB(t)$, and still denoting by $\beta$ its restriction to $[\tau, T]$, we have $\beta\in\cB(\tau)$. We define $\tilde\gamma\in\Gamma(t)$ as 
$$
\tilde\gamma[\beta](s)=\begin{cases}\bar\gamma[\beta](s),&s\in[t, \tau]\\ \gamma_{(m^X(\cdot, \tau; \bar \gamma[\beta], t, \bar m^X), m^Y(\cdot, \tau; \beta, t, \bar m^Y), \tau)}[\beta](s),&s\in[\tau, T]\end{cases}
$$
where $m^X(\cdot, \tau; \bar \gamma[\beta], t, \bar m^X)$ is the density at time $\tau$ of the solution to \eqref{contequationformassX} with $\a=\bar\gamma[\beta]$. The strategy $\tilde\gamma$ is well-defined and non-anticipating, i.e., it belongs to $\Gamma(t)$. So we have
\begin{multline*}
w(\bar m^X, \bar m^Y, t)\geq\sup_{\beta\in\cB(t)}\Bigg(\int_t^{\tau}\ell(m^X(\cdot, s; \beta, t, \bar m^X), m^Y(\cdot, s; \beta, t, \bar m^Y), s, \bar\gamma[\beta](s), \beta(\cdot, s))ds\\+\underbar V(m^X(\cdot, s; \beta, \tau, \bar m), m^Y(\cdot, \tau; \beta, t, \bar m), \tau)\Bigg)-\varepsilon\\
\geq\sup_{\beta\in\cB(t)}\Bigg(\int_t^{\tau}\ell(m^X(\cdot, s; \beta, t, \bar m^X), m^Y(\cdot, s; \beta, t, \bar m^Y), s, \bar\gamma[\beta](s), \beta(\cdot, s))\\+J(m^X(\cdot, s; \beta, t, \bar m^X), m^Y(\cdot, \tau; \beta, t, \bar m^Y), \tau, \gamma_{(m^X(\cdot, s; \beta, t, \bar m^X), m^Y(\cdot, \tau; \beta, t, \bar m^Y), \tau)}[\beta], \beta)\Bigg)-2\varepsilon\\
=\sup_{\beta\in\cB(t)}J(\bar m^X, \bar m^Y, t, \tilde\gamma[\beta], \beta)-2\varepsilon\geq\underbar V(\bar m^X, \bar m^Y, t)-2\varepsilon,
\end{multline*}
 and this concludes the proof of point $(i)$. 
\end{proof}
\begin{proposition}
Suppose that the assumptions on $\ell$, $\psi$ and \eqref{vectorfields} hold. Then the lower value function $\underbar V$ is bounded and uniformly continuous.
\end{proposition}
\begin{proof}
Since $\ell$ and $\psi$ are bounded function, we have that  for all $(\bar m^X, \bar m^Y, t)$ 
$$
\underbar V(\bar m^X, \bar m^Y, t)\leq G_1T+G_2,
$$
where $G_1>0$ and $G_2>0$ are the bounds for $\ell$ and $\psi$ respectively. So $\underbar V(\bar m^X, \bar m^Y, t)$ is bounded.
\par
Consider $(m^X_1, m^Y_1, t_1), (m^X_2, m^Y_2, t_2)\in  H^1(\mathbb R^d)\times H^1(\mathbb R^d)\times[0, T]$, and  $\varepsilon>0$. Let $\gamma_2\in\Gamma(t_2)$ be such that
$$
\underbar V(\bar m^X_2, \bar m^Y_2, t_2)\geq\sup_{\beta\in\cB(t_2)}J(\bar m^X_2, \bar m^Y_2, t_2, \gamma_2[\beta], \beta)-\varepsilon.
$$
We define $\gamma_2^1\in\Gamma(t_1)$ (and hence applying to controls in $\cB(t_1)$, which are defined on $[t_1,T]$) as
$$
\gamma_2^1[\beta](s)=\begin{cases}\gamma_2[\tilde \beta](s)&\text{for any }s\ \text{if }t_2\leq t_1\\
\bar a&\text{for }t_1\leq s\leq t_2\ \text{if }t_1\leq t_2\\
\gamma_2[\beta_{|[t_2, T]}](s)&\text{for }s\geq t_2\ \text{if }t_1\leq t_2
\end{cases},
$$
where $\bar a\in\tilde\cA$ is any a priori fixed constant (in time) control for the mass $X$ and, if $t_2\leq t_1$,
$$
\tilde \beta(\cdot, s)=\begin{cases}
\tilde b&\text{if }t_2\leq s\leq t_1\\
\beta(\cdot, s)&\text{if }s\geq t_1
\end{cases}.
$$
\noindent
where $\tilde b\in\tilde\cB$ is any a priori fixed constant (in time) control for the mass $Y$.
\par
Consider $\beta_1\in\cB(t_1)$ such that
$$
J(\bar m^X_1,\bar m^Y_1, t_1, \gamma_2^1[\beta_1], \beta_1)\geq\sup_{\beta\in\cB(t_1)}J(\bar m^X_1, \bar m^Y_1, t_1, \gamma_2^1[\beta], \beta)-\varepsilon.
$$
Moreover, let us take $\beta_1^2\in\cB(t_2)$ such that
$$
\beta_1^2(\cdot, s)=\beta_1(\cdot, s)\ \text{for }s\geq t_2+|t_1-t_2|.
$$
So we have
\begin{multline*}
\underbar V(\bar m^X_1,\bar m^Y_1, t_1)-\underbar V(\bar m^X_2, \bar m^Y_2, t_2)\\
\leq\sup_{\beta\in\cB(t_1)}J(\bar m^X_1, \bar m^Y_1, t_1, \gamma_2^1[\beta], \beta)-\sup_{\beta\in\cB(t_2)}J(\bar m^X_2, \bar m^Y_2, t_2, \gamma_2[\beta], \beta)+\varepsilon\\
\leq J(\bar m^X_1, \bar m^Y_1, t_1, \gamma_2^1[\beta_1], \beta_1)-J(\bar m^X_2, \bar m^Y_2, t_2, \gamma_2[\beta_1^2], \beta_1^2)+2\varepsilon.
\end{multline*}
Now, if 
\begin{equation}
\label{stimacosti}
|J(\bar m^X_1, \bar m^Y_1, t_1, \gamma_2^1[\beta_1], \beta_1)-J(\bar m^X_2, \bar m^Y_2, t_2, \gamma_2[\beta_1^2], \beta_1^2)|
\end{equation}
is infinitesimal as $|t_1-t_2|+\|\bar m^X_1(\cdot)-\bar m^X_2(\cdot)\|_{H^1(\mathbb R^d)}+\|\bar m^Y_1(\cdot)-\bar m^Y_2(\cdot)\|_{H^1(\mathbb R^d)}$, it goes to zero then we conclude. 
\par
We prove it in the case $t_1\leq t_2$. The case $t_2\leq t_1$ goes similarly. Let $\alpha \in\cA(t_1)$ and $\beta \in \cB(t_1)$, then \eqref{stimacosti} is equal to
\begin{multline*}
|J(\bar m^X_1, \bar m^Y_1, t_1, \gamma_2^1[\beta_1], \beta_1)-J(\bar m^X_2, \bar m^Y_2, t_2, \gamma_2[\beta_1^2], \beta_1^2)|=\\
\Bigg|\int_{t_1}^T\ell( m^X(\cdot, s; \alpha, t_1, \bar m^X_1), m^Y(\cdot, s; \beta, t_1, \bar m^Y_1), s, \gamma_2^1[\beta_1], \beta_1(\cdot, s))ds\\
+\psi(m^X(\cdot, T; \alpha, t_1, \bar m^X_1), m^Y(\cdot, T; \beta, t_1, \bar m^Y_1))\\
-\int_{t_2}^T\ell(m^X(\cdot, s; \alpha, t_2, \bar m^X_2), m^Y(\cdot, s; \beta, t_2, \bar m^Y_2), s, \gamma_2[\beta_1^2](s), \beta_1^2(\cdot, s))ds\\
-\psi(m^X(\cdot, T; \alpha, t_2, \bar m^X_2), m^Y(\cdot, T; \beta, t_2, \bar m^Y_2))\Bigg|\\
=\Bigg|\int_{t_1}^{t_2}\ell(m^X(\cdot, s; \alpha, t_1, \bar m^X_1), m^Y(\cdot, s; \beta, t_1, \bar m^Y_1), s, \bar a, \beta_1(\cdot, s))ds\\
+\int_{t_2}^T\ell(m^X(\cdot, s; \alpha, t_1, \bar m^X_1), m^Y(\cdot, s; \beta, t_1, m^Y_1), s, \gamma_2[\beta_{|[t_2, T]}](s), \beta_1^2(\cdot, s))ds\\
-\int_{t_2}^T\ell(m^X(\cdot, s; \alpha, t_2, \bar m^X_2), m^Y(\cdot, s; \beta, t_2, \bar m^Y_2), s, \gamma_2[\beta_{|[t_2, T]}](s), \beta_1^2(\cdot, s))ds\\
+\psi(m^X(\cdot, T; \alpha, t_1, \bar m^X_1), m^Y(\cdot, T; \beta, t_1, \bar m^Y_1)))-\psi(m^X(\cdot, T; \alpha, t_2, \bar m^X_2), m^Y(\cdot, T; \beta, t_2, \bar m^Y_2))\Bigg|.\\
\end{multline*}
Since $\ell$ and $\psi$ are uniformly strongly continuous, and using \eqref{dipendenzacont1} in Proposition \ref{propositionestimates}, we derive the following estimation 
\begin{multline*}
|J(\bar m^X_1, \bar m^Y_1, t_1, \gamma_2^1[\beta_1], \beta_1)-J(\bar m^X_2, \bar m^Y_2, t_2, \gamma_2[\beta_1^2], \beta_1^2)|\\
\leq G_1|t_1-t_2|+(\omega_{\psi}+T\omega_{\ell})\left(\tilde L\left(\|\bar m^X_1(\cdot)-\bar m^X_2(\cdot)\|_{H^1(\R^d)}+\|\bar m^Y_1(\cdot)-\bar m^Y_2(\cdot)\|_{H^1(\R^d)}\right)\right),
\end{multline*}
where $\tilde L$ is a suitable constant (see \eqref{dipendenzacont1} and also the comments after \eqref{controlsX}-\eqref{controlsY}).
Therefore, as $|t_1-t_2|+\|\bar m^X_1(\cdot)-\bar m^X_2(\cdot)\|_{H^1(\mathbb R^d)}+\|\bar m^Y_1(\cdot)-\bar m^Y_2(\cdot)\|_{H^1(\mathbb R^d)}$ goes to zero, \eqref{stimacosti} is infinitesimal. This concludes the proof.
\end{proof}
\subsection{The Hamilton-Jacobi-Isaacs equation for $\underbar V$}
\label{HJIforV}
In this section, we derive the corresponding Hamilton-Jacobi-Isaacs equation and we are going to consider it in a suitable set $\tilde\cM$ for the variables $(m^X, m^Y, t)$, which is compact in $L^2(\R^d)\times L^2(\R^d)\times[0, T]$ and invariant for the controlled evolutions $m^X(\cdot, s; \alpha, t, \bar m^X)$ and $m^Y(\cdot, s; \beta, t, \bar m^Y)$. Such a set is detected by an analogous construction as the one in \S\ref{sec:invariantset}, one per each population $X$ and $Y$. Indeed, there will be a set $\cM^X$ for the evolution $m^X(\cdot, s; \alpha, t, \bar m^X)$ and a set $\cM^Y$ for the evolution $m^Y(\cdot, s; \beta, t, \bar m^Y)$, and both of them are constructed with suitable constants and sets for each population. Therefore, the domain we are looking for is, for suitable compact sets $K^X(t)\subset H^1(\Omega_1^X(T))$ and $K^Y(t)\subset H^1(\Omega_1^Y(T))$ as in \S\ref{sec:invariantset},
\begin{multline}
\label{insiemeM}
\tilde\cM=\bigcup_{t\in[0, T]}\left(\cM^X(t)\times\cM^Y(t)\times\{t\}\right)\\
=\left\{(\bar m^X, \bar m^X, t)\in\left(W^{1, \infty}(\Omega_1^X(T))\cap K^X(t)\right)\times\left(W^{1, \infty}(\Omega_1^Y(T))\cap K^Y(t)\right)\times[0, T]:\right.\\
\left.\supp\bar m^X\subset\Omega_1^X(t),\ \supp\bar m^Y\subset\Omega_1^Y(t),\right.\\
\left. \|m^X(\cdot, s; \a, t, \bar m^X)\|_{W^{1, \infty}(\Omega_1^X(T))}\leq B^Xe^{M^XT}, \ \|m^Y(\cdot, s; \b, t, \bar m^Y)\|_{W^{1, \infty}(\Omega_1^Y(T))}\leq B^Ye^{M^YT}\right.\\
\left.\text{for all }s\in[t,T],\ \a\in\cA(t)\ \text{ and }\ \beta\in\cB(t)\right\}
\end{multline}
and it satisfies the above required properties. We also notice that, in general, the constants $B^X, B^Y$ as well as $M^X, M^Y$ may be different. 
\par\smallskip
We define the Hamiltonian function $H:\tilde\cM\times L^2(\R^d)\times L^2(\R^d)\lra\R$ as
$$
H(m^X, m^Y, t, p, q):=\min_{b\in\tilde\cB}\max_{a\in\tilde\cA}\left\{\langle p, \operatorname{div}(am^X)\rangle_{L^2(\R^d)}+\langle q, \operatorname{div}(bm^Y)\rangle_{L^2(\R^d)}-\ell(m^X, m^Y, t, a, b)\right\}.
$$
In the following, we will consider the fields $\a(\cdot, s)$ and $\beta(\cdot, s)$ defined on the compact sets $\Omega_1^X(T)$ and $\Omega_1^Y(T)$ (see \S\ref{sec:invariantset}), i.e., $\a:\Omega_1^X(T)\times[t, T]\lra\R^d$ and $\beta:\Omega_1^Y(T)\times[t, T]\lra\R^d$, respectively. Additionally,  $D_X$ and $D_Y$ will denote the Fr\'echet differentials with respect to $m^X$ and $m^Y$, respectively.
\begin{definition}
A function $u\in C^0(\tilde\cM)$ is a viscosity subsolution of 
\begin{equation}
\label{hjbviscosity}
-u_t(m^X, m^Y, t)+H(m^X, m^Y, t, D_Xu, D_Yu)=0\quad\text{in }\tilde\cM
\end{equation}
if, for any $\phi\in C^1(L^2(\R^d)\times L^2(\R^d)\times[0, T])$,
$$
-\phi_t(\bar m^X, \bar m^Y, \bar t)+H(\bar m^X, \bar m^Y, \bar t, D_X\phi(\bar m^X, \bar m^Y, \bar t), D_Y\phi(\bar m^X, \bar m^Y, \bar t))\leq0
$$
at any local maximum point $(\bar m^X, \bar m^Y, \bar t)\in\tilde\cM$ of $u-\phi$. Similarly, $u\in C^0(\tilde\cM)$ is a viscosity supersolution of \eqref{hjbviscosity} if, for any $\phi\in C^1(L^2(\R^d)\times L^2(\R^d)\times[0, T])$,
$$
-\phi_t(\tilde m^X, \tilde m^Y, \tilde t)+H(\tilde m^X, \tilde m^Y, \tilde t, D_X\phi(\tilde m^X, \tilde m^Y, \tilde t), D_Y\phi(\tilde m^X, \tilde m^Y, \tilde t))\geq0
$$
at any local minimum point $(\tilde m^X, \tilde m^Y, \tilde t)\in\tilde \cM$ of $u-\phi$. Finally, $u$ is a viscosity solution of \eqref{hjbviscosity} if it is simultaneously a viscosity sub- and supersolution. 
\end{definition}
We observe that the above local maximum/minimum point is w.r.t. $\tilde\cM$ and not necessarily w.r.t. $L^2(\R^d)\times L^2(\R^d)\times[0, T]$, where the test function is defined. This is not a problem. As discussed in \S\ref{sec:invariantset}, $\tilde \cM$ is invariant for our evolutions and moreover, in $\tilde \cM$ the convergences of $m^X$ and $m^Y$ in $H^1(\R^d)$ and in $L^2(\R^d)$ are equivalent.
\begin{theorem}
\label{esistenza}
Under the previous assumptions, the lower value function $\underbar V$ is a viscosity solution of
\begin{equation}
\label{HJI}
\begin{cases}
-V_t+H(m^X, m^Y, t, D_XV, D_YV)=0&\text{in }\tilde \cM\\
V(m^X, m^Y, T)=\psi(m^X, m^Y),&(m^X, m^Y)\in\cM^X(T)\times \cM^Y(T)\times\{T\}
\end{cases}.
\end{equation}
\end{theorem}
\begin{proof}
We will prove only that $\underbar V$ is a supersolution of \eqref{HJI}. The proof that it is also a subsolution requires a suitable lemma, which is provided below as Lemma \ref{lemmaxstrategia} (see also Remark \ref{remark1}). This lemma is formulated for our infinite-dimensional case with two masses. The proof of the lemma can be easily derived from the corresponding result in \cite{BCM} for the case with one mass and a single player.
\par
Take $\phi\in C^1(L^2(\R^d)\times L^2(\R^d)\times[0, T])$. Let $(\bar m^X, \bar m^Y, \bar t)$ be a local minimum point for $\underbar V-\phi$, and $\underbar V(\bar m^X, \bar m^Y, \bar t)=\phi(\bar m^X, \bar m^Y, \bar t)$. By contradiction, we suppose that
$$
-\phi_t(\bar m^X, \bar m^Y, \bar t)+H(\bar m^X, \bar m^Y, \bar t, D_X\underbar V, D_Y\underbar V)=-\theta<0.
$$
By the definition of $H$, there exists $b^*\in\tilde\cB$ such that
\begin{multline*}
-\phi_t(\bar m^X, \bar m^Y, \bar t)+\langle D_X\phi(\bar m^X, \bar m^Y, \bar t), \div(a\bar m^X)\rangle_{L^2(\Omega_1^X(T))}\\
+\langle D_Y\phi(\bar m^X, \bar m^Y, \bar t), \div(b^*\bar m^Y)\rangle_{L^2(\Omega_1^Y(T))}-\ell(\bar m^X, \bar m^Y, \bar t, a, b^*)\leq -\theta
\end{multline*}
for all $a\in A$. For $\tau$ sufficiently close to $\bar t$ and any $\gamma\in\Gamma(\bar t)$, we have
\begin{multline*}
-\phi_t(m^X(\cdot, s; \gamma[b^*], \bar t, \bar m^X), m^Y(\cdot, s; b^*, \bar t, \bar m^Y), s)\\
+\langle D_X\phi(m^X(\cdot, s; \gamma[b^*], \bar t, \bar m^X), m^Y(\cdot, s; b^*, \bar t, \bar m^Y), s), \div(\gamma[b^*]m^X(\cdot, s; \gamma[b^*], \bar t, \bar m^X))\rangle_{L^2(\Omega_1^X(T))}\\
+\langle D_Y\phi(m^X(\cdot, s; \gamma[b^*], \bar t, \bar m^X), m^Y(\cdot, s; b^*, \bar t, \bar m^Y), s), \div (b^*m^Y(\cdot, s; b^*, \bar t, \bar m^Y))\rangle_{L^2(\Omega_1^Y(T))}\\
-\ell(m^X(\cdot, s; \gamma[b^*], \bar t, \bar m^X), m^Y(\cdot, s; b^*, \bar t, \bar m^Y), s, \gamma[b^*](s), b^*)\leq-\frac{\theta}{2}
\end{multline*}
for every $\bar t\leq s<\tau$. Integrating from $\bar t$ to $\tau$, and applying \eqref{contequationformassX} and \eqref{contequationformassY}, yields
\begin{multline*}
\phi(\bar m^X, \bar m^Y, \bar t)-\phi(m^X(\cdot, \tau; \gamma[b^*], \bar t, \bar m^X), m^Y(\cdot, \tau; b^*, \bar t, \bar m^Y), \tau)\\
-\int_{\bar t}^{\tau}\ell(m^X(\cdot, s; \gamma[b^*], \bar t, \bar m^X), m^Y(\cdot, s; b^*, \bar t, \bar m^Y), s, \gamma[b^*](s), b^*)ds\leq -\frac{\theta(\tau-\bar t)}{4},
\end{multline*}
for $\tau$ sufficiently close to $\bar t$. From
\begin{multline*}
\phi(\bar m^X, \bar m^Y, \bar t)-\phi(m^X(\cdot, \tau; \gamma[b^*], \bar t, \bar m^X), m^Y(\cdot, \tau; b^*, \bar t, \bar m^Y), \tau)\\
\geq \underbar V(\bar m^X, \bar m^Y, \bar t)-\underbar V(m^X(\cdot, \tau, \gamma[b^*], \bar t, \bar m^X), m^Y(\cdot, \tau; b ^*, \bar t, \bar m^Y), \tau),
\end{multline*}
one has
\begin{multline*}
\underbar V(m^X(\cdot, \tau, \gamma[b^*], \bar t, \bar m^X), m^Y(\cdot, \tau; b^*, \bar t, \bar m^Y), \tau)\\
+\int_{\bar t}^{\tau}\ell(m^X(\cdot, s; \gamma[b^*], \bar t, \bar m^X), m^Y(\cdot, s; b^*, \bar t, \bar m^Y), s, \gamma[b^*](s), b^*)ds\\
\geq \frac{\theta(\tau-\bar t)}{4}+\underbar V(\bar m^X, \bar m^Y, \bar t).
\end{multline*}
Therefore, we deduce that
\begin{multline*}
\inf_{\gamma\in\Gamma(\bar t)}\sup_{\beta\in\cB(\bar t)}\Bigg\{\int_{\bar t}^{\tau}\ell(m^X(\cdot, s; \gamma[\beta], \bar t, \bar m^X), m^Y(\cdot, s; \beta, \bar t, \bar m^Y), s, \gamma[\beta](s), \beta(\cdot, s))ds\\
+\underbar V(m^X(\cdot, \tau; \gamma[\beta], \bar t, \bar m^X), m^Y(\cdot, \tau; \beta, \bar t, \bar m^Y), \tau)\Bigg\}>\underbar V(\bar m^X, \bar m^Y, \bar t),
\end{multline*}
which contradicts \eqref{dpp}.
\end{proof}

\begin{lemma}
\label{lemmaxstrategia}
Assume the hypotheses of Theorem \ref{esistenza}. Let $(m^X, m^Y, t)\in\tilde\cM$ and $\phi\in C^1(L^2(\R^d)\times L^2(\R^d)\times[0, T])$ be such that
$$
-\phi_t(m^X, m^Y, t)+H(m^X, m^Y, t, D_X\phi(m^X, m^Y, t), D_Y\phi(m^X, m^Y, t))=\theta>0.
$$
Then there exists $\gamma^*\in\Gamma(t)$ such that for all $\b\in\cB(t)$ and $\tau>t$ sufficiently close to $t$,
\begin{multline*}
\int_t^{\tau}\Big\{\ell(m^X(\cdot, s; \gamma^*[\beta], t, \bar m^X), m^Y(\cdot, s; \beta, t, \bar m^Y), s, \gamma^*[\beta](s), \beta(\cdot, s))\\
-\langle D_X\phi(m^X(\cdot, s; \gamma^*[\beta], t, \bar m^X), m^Y(\cdot, s; \beta, t, \bar m^Y), s), \operatorname{div}(\gamma^*[\beta](s)m^X(\cdot, s;\gamma^*[\beta], t, \bar m^X)\rangle_{L^2(\Omega_1^X(T))}\\
-\langle D_Y\phi(m^X(\cdot, s; \gamma^*[\beta], t, \bar m^X), m^Y(\cdot, s; \beta, t, \bar m^Y), s), \operatorname{div}(\beta(\cdot, s)m^Y(\cdot, s; \beta, t, \bar m^Y))\rangle_{L^2(\Omega_1^Y(T))}\\
+\phi_t(m^X(\cdot, s; \gamma^*[\beta], t, \bar m^X), m^Y(\cdot, s; \beta, t, \bar m^Y), s)\Big\}ds\leq-\frac{\theta(\tau-t)}{4}.
\end{multline*}
\end{lemma}
In the next result, we show that the lower value function of the problem is the unique viscosity solution of the HJI problem \eqref{HJI}.
\begin{theorem}
\label{teoconfronto}
Assume the hypotheses of Theorem \ref{esistenza}. Let $u_1, u_2$ be bounded and uniformly continuous functions and, respectively, viscosity sub- and supersolution of
$$
-V_t+H(m^X, m^Y, t, D_XV, D_YV)=0\quad\text{in }\tilde\cM.
$$
If $u_1\leq u_2$ on $\cM^X(T)\times \cM^Y(T)\times\{T\}$, then $u_1\leq u_2$ in $\tilde\cM$. In particular, this guarantees the uniqueness of the bounded and uniformly continuous viscosity solution to \eqref{HJI}, which is $\underbar V$.
\end{theorem}
For the proof, we need the following lemma.
\begin{lemma}
\label{lemmaham}
For all $\zeta, \xi>0$, $(m_1^X, m_1^Y, t_1), (m_2^X, m_2^Y, t_2)\in\tilde\cM$, $p=\frac{2(m_1^X-m_2^X)}{\zeta^2}$ and $q=\frac{2(m_1^Y-m_2^Y)}{\xi^2}$,  it holds
\begin{multline*}
\left|H\left(m_1^X, m_1^Y, t_1, \frac{2(m_1^X-m_2^X)}{\zeta^2}, \frac{2(m_1^Y-m_2^Y)}{\xi^2}\right)-H\left(m_2^X, m_2^Y, t_2, \frac{2(m_1^X-m_2^X)}{\zeta^2}, \frac{2(m_1^Y-m_2^Y)}{\xi^2}\right)\right|\\
\leq M\left(\frac{\|m_1^X-m_2^X\|^2_{L^2(\Omega_1^X(T))}}{\zeta^2}+\frac{\|m_1^Y-m_2^Y\|^2_{L^2(\Omega_1^Y(T))}}{\xi^2}\right)\\
+\omega_{\ell}\left(\|m_1^X-m_2^X\|_{H^1(\Omega_1^X(T))}+\|m_1^Y-m_2^Y\|_{H^1(\Omega_1^Y(T))}+|t_1-t_2|\right),
\end{multline*}
where $M=\max\{M^X, M^Y\}$ (see \eqref{insiemeM} and comments after \eqref{controlsX}-\eqref{controlsY}) and $\omega_{\ell}$ is the modulus of continuity of the running cost $\ell$.
\end{lemma}
\begin{proof}
Let $b'\in\tilde\cB$ be such that
\begin{multline*}
H\left(m_1^X, m_1^Y, t_1, \frac{2(m_1^X-m_2^X)}{\zeta^2}, \frac{2(m_1^Y-m_2^Y)}{\xi^2}\right)\\
\geq \left\langle\frac{2(m_1^X-m_2^X)}{\zeta^2}, \div(am_1^X)\right\rangle_{L^2(\Omega_1^X(T))}+\left\langle \frac{2(m_1^Y-m_2^Y)}{\xi^2}, \div(b'm_1^Y)\right\rangle_{L^2(\Omega_1^Y(T))}
-\ell(m_1^X, m_1^Y, t_1, a, b')
\end{multline*}
for every $a\in\tilde\cA$. Then, we consider $a'\in\tilde\cA$ such that
\begin{align*}
&H\left(m_2^X, m_2^Y, t_2, \frac{2(m_1^X-m_2^X)}{\zeta^2}, \frac{2(m_1^Y-m_2^Y)}{\xi^2}\right)\\
&\leq \left\langle \frac{2(m_1^X-m_2^X)}{\zeta^2}, \div(a'm_2^X)\right\rangle_{L^2(\Omega_1^X(T))}+\left\langle \frac{2(m_1^Y-m_2^Y)}{\xi^2}, \div(b'm_2^Y)\right\rangle_{L^2(\Omega_1^Y(T))}
-\ell(x_2, m_2, t_2, a', b').
\end{align*}
Combining the above estimates, we deduce the following
\begin{multline*}
\left|H\left(m_2^X, m_2^Y, t_2, \frac{2(m_1^X-m_2^X)}{\zeta^2}, \frac{2(m_1^Y-m_2^Y)}{\xi^2}\right)-H\left(m_1^X, m_1^Y, t_1, \frac{2(m_1^X-m_2^X)}{\zeta^2}, \frac{2(m_1^Y-m_2^Y)}{\xi^2}\right)\right|\\
\leq\Bigg|\left\langle \frac{2(m_1^X-m_2^X)}{\zeta^2}, \div(a'm_2^X)\right\rangle_{L^2(\Omega_1^X(T))}+\left\langle \frac{2(m_1^Y-m_2^Y)}{\xi^2}, \div(b'm_2^Y)\right\rangle_{L^2(\Omega_1^Y(T))}
-\ell(m_2^X, m_2^Y, t_2, a', b')\\
-\left\langle\frac{2(m_1^X-m_2^X)}{\zeta^2}, \div(a'm_1^X)\right\rangle_{L^2(\Omega_1^X(T))}-\left\langle \frac{2(m_1^Y-m_2^Y)}{\xi^2}, \div(b'm_1^Y)\right\rangle_{L^2(\Omega_1^Y(T))}
+\ell(m_1^X, m_1^Y, t_1, a', b')\Bigg|\\
\leq\Bigg|\left\langle \frac{2(m_1^X-m_2^X)}{\zeta^2}, \div(a'(m_2^X-m_1^X))\right\rangle_{L^2(\Omega_1^X(T))}
+\left\langle\frac{2(m_1^Y-m_2^Y)}{\xi^2}, \div(b'(m_2^Y-m_1^Y))\right\rangle_{L^2(\Omega_1^Y(T))}\\
+\omega_{\ell}\left(\|m_1^X-m_2^X\|_{H^1(\Omega_1^X(T))}+\|m_1^Y-m_2^Y\|_{H^1(\Omega_1^Y(T))}+|t_1-t_2|\right)\Bigg|\\
\leq \left|\frac{1}{\zeta^2}\int_{\Omega_1^X(T)}\div(a')(m_1^X-m_2^X)^2dx\right|+\left|\frac{1}{\xi^2}\int_{\Omega_1^Y(T)}\div(b')(m_1^Y-m_2^Y)^2dx\right|\\
+\omega_{\ell}\left(\|m_1^X-m_2^X\|_{H^1(\Omega_1^X(T))}+\|m_1^Y-m_2^Y\|_{H^1(\Omega_1^Y(T))}+|t_1-t_2|\right)\\
\leq M\left(\frac{\|m_1^X-m_2^X\|^2_{L^2(\Omega_1^X(T))}}{\zeta^2}+\frac{\|m_1^Y-m_2^Y\|^2_{L^2(\Omega_1^Y(T))}}{\xi^2}\right)\\
+\omega_{\ell}\left(\|m_1^X-m_2^X\|_{H^1(\Omega_1^X(T))}+\|m_1^Y-m_2^Y\|_{H^1(\Omega_1^Y(T))}+|t_1-t_2|\right),
\end{multline*}
where in the second-to-last inequality integration by parts was used, along with the fact  that $m_1^X,m_2^X$ and $m_1^Y,m_2^Y$ vanish at $\partial\Omega_1^X(T)$ and $\partial\Omega_1^Y(T)$, respectively.
\end{proof}
{\it Proof (of Theorem \ref{teoconfronto}).} Let $G=\sup_{\tilde\cM}(u_1-u_2)$. We aim to show that $G \leq 0$. To do so, we proceed by contradiction, assuming instead that $G>0$.
\par
To simplify the proof, we replace $u_1$ with 
$$
(u_1(m^X, m^Y, t))_{\eta}:=u_1(m^X, m^Y, t)-\eta t
$$
for some $\eta>0$ sufficiently small. This allows us to assume, without loss of generality, that $u_1$ is a strict subsolution of \eqref{HJI} since $(u_1)_{\eta}$ is a subsolution of 
\begin{equation}
\label{eqreduced}
-\frac{\partial}{\partial t}(u_1)_{\eta}+H(m^X, m^Y, t, D_X(u_1)_{\eta}, D_Y(u_1)_{\eta})\leq-\eta<0\quad\text{in}\ \tilde\cM.
\end{equation}
It will be sufficient to show that $(u_1)_{\eta}\leq u_2$ in $\tilde\cM$ for any $\eta$, and then take the limit  as $\eta \rightarrow 0$. Note that $(u_1)_{\eta}\leq u_2$ also holds on $\cM^X(T)\times\cM^Y(T)\times\{T\}$. For simplicity, we will omit $\eta$ and use $u_1$ in place  of $(u_1)_{\eta}$. Here, there is a difficulty with $\cM^X(0)\times\cM^Y(0)\times\{0\}$ that can be addressed by adapting \cite[Lemma 4.6]{BCM} to the context of two masses.
\par
Introduce the following test function
\begin{multline*}
\Psi_{\ep, \xi, \a}(m^X_1, m^Y_1, t_1, m^X_2, m^Y_2, t_2)=u_1(m_1^X, m_1^Y, t_1)-u_2(m_2^X, m_2^Y, t_2)\\
-\frac{\|m_1^X-m_2^X\|^2_{L^2(\Omega_1^X(T))}}{\ep^2}-\frac{\|m_1^Y-m_2^Y\|^2_{L^2(\Omega_1^Y(T))}}{\xi^2}-\frac{|t_1-t_2|^2}{\a^2}.
\end{multline*}
Since $\Psi_{\ep, \xi, \a}$ is continuous on $\tilde\cM\times\tilde\cM$, there exists a maximum point $(\tilde m^X_1,\tilde m^Y_1,\tilde t_1,\tilde m^X_2,\tilde m^Y_2,\tilde t_2)$, and we set $\bar G:=\Psi_{\ep, \xi, \a}(\tilde m^X_1,\tilde m^Y_1,\tilde t_1,\tilde m^X_2,\tilde m^Y_2,\tilde t_2)$. We prove the following:
\begin{itemize}
\item[$(1)$] when $\ep, \xi, \s\to0$, then $\bar G\to G$;
\item[$(2)$] $u_1(\tilde m_1^X, \tilde m_1^Y, \tilde t_1)-u_2(\tilde m_2^X, \tilde m_2^Y, \tilde t_2)\to G$ as $\ep, \xi, \a\to0$;
\item[$(3)$] it holds
$$
\frac{\|m_1^X-m_2^X\|^2_{L^2(\Omega_1^X(T))}}{\ep^2},\frac{\|m_1^Y-m_2^Y\|^2_{L^2(\Omega_1^Y(T))}}{\xi^2},\frac{|t_1-t_2|^2}{\a^2}\to0\quad\text{as }\ \ep, \xi, \a\to0;
$$
\item[$(4)$] $(\tilde m_1^X, \tilde m_1^Y, \tilde t_1), (\tilde m_2^X, \tilde m_2^Y, \tilde t_2)\in\tilde\cM$.
\end{itemize}
Whereas $(\tilde m^X_1,\tilde m^Y_1,\tilde t_1,\tilde m^X_2,\tilde m^Y_2,\tilde t_2)$ is a maximum point of $\Psi_{\ep, \xi, \a}$, we have
\begin{multline*}
u_1(m_1^X, m_1^Y, t_1)-u_2(m_2^X, m_2^Y, t_2)
-\frac{\|m_1^X-m_2^X\|^2_{L^2(\Omega_1^X(T))}}{\ep^2}-\frac{\|m_1^Y-m_2^Y\|^2_{L^2(\Omega_1^Y(T))}}{\xi^2}-\frac{|t_1-t_2|^2}{\a^2}\\
\leq u_1(\tilde m_1^X,\tilde m_1^Y,\tilde t_1)-u_2(\tilde m_2^X,\tilde m_2^Y,\tilde t_2)-\frac{\|\tilde m_1^X-\tilde m_2^X\|^2_{L^2(\Omega_1^X(T))}}{\ep^2}-\frac{\|\tilde m_1^Y-\tilde m_2^Y\|^2_{L^2(\Omega_1^Y(T))}}{\xi^2}-\frac{|\tilde t_1-\tilde t_2|^2}{\a^2}\\
=\bar G
\end{multline*}
for any $(m_1^X, m_1^Y, t_1), (m_2^X, m_2^Y, t_2)\in\tilde\cM$. Let us choose $m_1^X=m_2^X$, $m_1^Y=m_2^Y$ and $t_1=t_2$ in the left-hand side to obtain
$$
u_1(m_1^X, m_1^Y, t_1)-u_2(m_1^X, m_1^Y, t_1)\leq\bar G\quad\text{for all }(m_1^X, m_1^Y, t_1)\in\tilde\cM,
$$
and, by taking the supremum over $(m_1^X, m_1^Y, t_1)$, we get $G\leq\bar G$.
\par
Using the boundedness of  $u_1$ and $u_2$ and setting $R:=\max\{\|u_1\|_{\infty}, \|u_2\|_{\infty}\}$, we have
\begin{multline*}
G\leq u_1(\tilde m_1^X,\tilde m_1^Y,\tilde t_1)-u_2(\tilde m_2^X,\tilde m_2^Y,\tilde t_2)
-\frac{\|\tilde m_1^X-\tilde m_2^X\|^2_{L^2(\Omega_1^X(T))}}{\ep^2}-\frac{\|\tilde m_1^Y-\tilde m_2^Y\|^2_{L^2(\Omega_1^Y(T))}}{\xi^2}-\frac{|\tilde t_1-\tilde t_2|^2}{\a^2}\\
\leq 2R-\frac{\|\tilde m_1^X-\tilde m_2^X\|^2_{L^2(\Omega_1^X(T))}}{\ep^2}-\frac{\|\tilde m_1^Y-\tilde m_2^Y\|^2_{L^2(\Omega_1^Y(T))}}{\xi^2}-\frac{|\tilde t_1-\tilde t_2|^2}{\a^2}.
\end{multline*}
Recalling our assumption on $G$, i.e., $G>0$, we deduce
\begin{equation}
\label{minoreR}
\frac{\|\tilde m_1^X-\tilde m_2^X\|^2_{L^2(\Omega_1^X(T))}}{\ep^2}+\frac{\|\tilde m_1^Y-\tilde m_2^Y\|^2_{L^2(\Omega_1^Y(T))}}{\xi^2}+\frac{|\tilde t_1-\tilde t_2|^2}{\a^2}\leq 2R.
\end{equation}
In particular, one has that $\|\tilde m_1^X-\tilde m_2^X\|_{L^2(\Omega_1^X(T))}, \|\tilde m_1^Y-\tilde m_2^Y\|_{L^2(\Omega_1^Y(T))}, |\tilde t_1-\tilde t_2|\to0$ as $\ep, \xi, \a\to0$.
\par
Due to the compactness of $\tilde\cM$ in $L^2(\R^d)\times L^2(\R^d)\times[0, T]$, we can assume without loss of generality that the points $(\tilde m_1^X, \tilde m_1^Y, \tilde t_1), (\tilde m_2^X, \tilde m_2^Y, \tilde t_2)$ converge to the same point. As a consequence of this fact the following inequalities hold
$$
\liminf (u_1(\tilde m_1^X, \tilde m_1^Y, \tilde t_1)-u_2(\tilde m_2^X, \tilde m_2^Y, \tilde t_2))\leq\limsup(u_1(\tilde m_1^X, \tilde m_1^Y, \tilde t_1)-u_2(\tilde m_2^X, \tilde m_2^Y, \tilde t_2))\leq G
$$
and
\begin{multline}
\label{disugmax}
G\leq u_1(\tilde m_1^X, \tilde m_1^Y, \tilde t_1)-u_2(\tilde m_2^X, \tilde m_2^Y, \tilde t_2)-\frac{\|\tilde m_1^X-\tilde m_2^X\|^2_{L^2(\Omega_1^X(T))}}{\ep^2}-\frac{\|\tilde m_1^Y-\tilde m_2^Y\|^2_{L^2(\Omega_1^Y(T))}}{\xi^2}-\frac{|\tilde t_1-\tilde t_2|^2}{\a^2}\\
\leq u_1(\tilde m_1^X, \tilde m_1^Y, \tilde t_1)-u_2(\tilde m_2^X, \tilde m_2^Y, \tilde t_2).
\end{multline}
Therefore, we deduce that $\lim (u_1(\tilde m_1^X, \tilde m_1^Y, \tilde t_1)-u_2(\tilde m_2^X, \tilde m_2^Y, \tilde t_2))=G$. Moreover, thanks to \eqref{disugmax},
\begin{multline*}
\bar G=u_1(\tilde m_1^X, \tilde m_1^Y, \tilde t_1)-u_2(\tilde m_2^X, \tilde m_2^Y, \tilde t_2)\\
-\frac{\|\tilde m_1^X-\tilde m_2^X\|^2_{L^2(\Omega_1^X(T))}}{\ep^2}-\frac{\|\tilde m_1^Y-\tilde m_2^Y\|^2_{L^2(\Omega_1^Y(T))}}{\xi^2}-\frac{|\tilde t_1-\tilde t_2|^2}{\a^2}\to G
\end{multline*}
and, since $u_1(\tilde m_1^X, \tilde m_1^Y, \tilde t_1)-u_2(\tilde m_2^X, \tilde m_2^Y, \tilde t_2)\to G$, we immediately get
$$
-\frac{\|\tilde m_1^X-\tilde m_2^X\|^2_{L^2(\Omega_1^X(T))}}{\ep^2}-\frac{\|\tilde m_1^Y-\tilde m_2^Y\|^2_{L^2(\Omega_1^Y(T))}}{\xi^2}-\frac{|\tilde t_1-\tilde t_2|^2}{\a^2}\to0.
$$
This concludes the proof of $(1)$, $(2)$ and $(3)$.
\par
For $(4)$, we observe that, if $(m_1^X, m_1^Y, t_1)$ is the limit of a subsequence of $(\tilde m_1^X, \tilde m_1^Y, \tilde t_1),(\tilde m_2^X, \tilde m_2^Y, \tilde t_2)$, then $u_1(m_1^X, m_1^Y, t_1)-u_2(m_1^X, m_1^Y, t_1)=M>0$ and therefore $(m_1^X, m_1^Y, t_1)$ cannot be on $\cM^X(T)\times \cM^Y(T)\times\{T\}$.
\par
Suppose that $\ep, \xi, \a$ are sufficiently small and such that $(4)$ holds. Since $(\tilde m_1^X, \tilde m_1^Y, \tilde t_1, \tilde m_2^X, \tilde m_2^Y, \tilde t_2)$ is a maximum point of $\Psi_{\ep, \xi, \a}$, then $(\tilde m_1^X, \tilde m_1^Y, \tilde t_1)$ is a maximum point of the following function
$$
(m_1^X, m_1^Y, t_1)\longmapsto u_1(m_1^X, m_1^Y, t_1)-\phi^1(m_1^X, m_1^Y, t_1),
$$
where
\begin{multline*}
\phi^1(m_1^X, m_1^Y, t_1)=u_2(\tilde m_2^X, \tilde m_2^Y, \tilde t_2)
+\frac{\|m_1^X-\tilde x_2^X\|^2_{L^2(\Omega_1^X(T))}}{\ep^2}+\frac{\|m_1^Y-\tilde m_2^Y\|^2_{L^2(\Omega_1^Y(T))}}{\xi^2}+\frac{|t_1-\tilde t_2|^2}{\a^2},
\end{multline*}
but $u_1$ is a viscosity subsolution of \eqref{eqreduced} and $(\tilde m_1^X, \tilde m_1^Y, \tilde t_1)\in\tilde\cM$. Hence, we have that
\begin{multline*}
-\phi^1_t(\tilde m_1^X, \tilde m_1^Y, \tilde t_1)+H(\tilde m_1^X, \tilde m_1^Y, \tilde t_1, D_X\phi^1(\tilde m_1^X, \tilde m_1^Y, \tilde t_1), D_Y\phi^1(\tilde m_1^X, \tilde m_1^Y, \tilde t_1))\\
=\frac{2(\tilde t_2-\tilde t_1)}{\a^2}+H\left(\tilde m_1^X, \tilde m_1^Y, \tilde t_1, \frac{2(\tilde m_1^X-\tilde m_2^X)}{\ep^2}, \frac{2(\tilde m_1^Y-\tilde m_2^Y)}{\xi^2}\right)\leq-\eta.
\end{multline*}
Similarly, $(\tilde m_2^X, \tilde m_2^Y, \tilde t_2)$ is a maximum point of the function
$$
(m_2^X, m_2^Y, t_2)\longmapsto-u_2(m_2^X, m_2^Y, t_2)+\phi^2(m_2^X, m_2^Y, t_2),
$$
where
\begin{multline*}
\phi^2(m_2^X, m_2^Y, t_2)=u_1(\tilde m_1^X, \tilde m_1^Y, \tilde t_1)-\frac{\|\tilde m_1^X-m_2^X\|^2_{L^2(\Omega_1^X(T))}}{\ep^2}-\frac{\|\tilde m_1^Y-m_2^Y\|^2_{L^2(\Omega_1^Y(T))}}{\xi^2}-\frac{|\tilde t_1-t_2|^2}{\a^2},
\end{multline*}
but $u_2$ is a viscosity supersolution of \eqref{HJI} and $(\tilde m_2^X, \tilde m_2^Y, \tilde t_2)\in\tilde\cM$. Then
\begin{multline*}
-\phi^2_t(\tilde m_2^X, \tilde m_2^Y, \tilde t_2)+H(\tilde m_2^X, \tilde m_2^Y, \tilde t_2, D_X\phi^2(\tilde m_2^X, \tilde m_2^Y, \tilde t_2), D_Y\phi^2(\tilde m_2^X, \tilde m_2^Y, \tilde t_2))\\
=\frac{2(\tilde t_2-\tilde t_1)}{\a^2}+H\left(\tilde m_2^X, \tilde m_2^Y, \tilde t_2, \frac{2(\tilde m_1^X-\tilde m_2^X)}{\ep^2}, \frac{2(\tilde m_1^Y-\tilde m_2^Y)}{\xi^2}\right)\geq0.
\end{multline*}
By subtracting the two viscosity inequalities, we obtain
\begin{multline*}
H\left(\tilde m_1^X, \tilde m_1^Y, \tilde t_1, \frac{2(\tilde m_1^X-\tilde m_2^X)}{\ep^2}, \frac{2(\tilde m_1^Y-\tilde m_2^Y)}{\xi^2}\right)\\
-H\left(\tilde m_2^X, \tilde m_2^Y, \tilde t_2, \frac{2(\tilde m_1^X-\tilde m_2^X)}{\ep^2}, \frac{2(\tilde m_1^Y-\tilde m_2^Y)}{\xi^2}\right)\leq-\eta.
\end{multline*}
Now we use Lemma \ref{lemmaham} to get
\begin{multline*}
H\left(\tilde m_1^X, \tilde m_1^Y, \tilde t_1, \frac{2(\tilde m_1^X-\tilde m_2^X)}{\ep^2}, \frac{2(\tilde m_1^Y-\tilde m_2^Y)}{\xi^2}\right)-H\left(\tilde m_2^X, \tilde m_2^Y, \tilde t_2, \frac{2(\tilde m_1^X-\tilde m_2^X)}{\ep^2}, \frac{2(\tilde m_1^Y-\tilde m_2^Y)}{\xi^2}\right)\\
\leq M\left(\frac{\|\tilde m_1^X-\tilde m_2^X\|_{L^2(\Omega_1^X(T))}}{\ep^2}+\frac{\|\tilde m_1^Y-\tilde m_2^Y\|^2_{L^2(\Omega_1^Y(T))}}{\xi^2}\right)\\
+\omega_{\ell}\left(\|\tilde m_1^X-\tilde m_2^X\|_{H^1(\Omega_1^X(T))}+\|\tilde m_1^Y-\tilde m_2^Y\|_{H^1(\Omega_1^Y(T))}+|\tilde t_1-\tilde t_2|\right)\leq-\eta.
\end{multline*}
But, on one side, $\|\tilde m_1^X-\tilde x_2^X\|_{L^2(\Omega_1^X(T))}, \|\tilde m_1^Y-\tilde m_2^Y\|_{L^2(\Omega_1^Y(T))}, |\tilde t_1-\tilde t_2|\to0$ as $\ep, \xi, \a\to0$ (see \eqref{minoreR}) and hence also $\|\tilde m_1^X-\tilde m_2^X\|_{H^1(\Omega_1^X(T))}\to0$ and $\|\tilde m_1^Y-\tilde m_2^Y\|_{H^1(\Omega_1^Y(T))}\to0$ as $\xi\to0$ since \eqref{modcontm} holds. On the other side
$$
M\left(\frac{\|\tilde m_1^X-\tilde m_2^X\|_{L^2(\Omega_1^X(T))}}{\ep^2}+\frac{\|\tilde m_1^Y-\tilde m_2^Y\|^2_{L^2(\Omega_1^Y(T))}}{\xi^2}\right)\to0\quad\text{when }\ep,\xi\to0.
$$
The above inequality gives a contradiction.
$\hfill\square$

\section{Two one-dimensional examples}
\label{sec:examples}
In this section, we argue about two one-dimensional examples ($d=1$).
\par\medskip
Following the one-dimensional examples reported in \cite{BCM}, here we may consider the case where the running cost $\ell$ is zero and the final cost is given by
\begin{equation}
\label{eq:g_1}
\psi_1(\bar m^X,\bar m^Y)=\int_{\mathbb{R}}\bar m^X(x)\bar m^Y(x)dx
\end{equation}
or, for some a-priori fixed $\delta>0$,
\begin{equation}
\label{eq:g_2}
\psi_2(\bar m^X,\bar m^Y)=\left(\int_{\tilde m^X-\delta}^{\tilde m^X+\delta}\bar m^Y(x)dx-\int_{\tilde m^Y-\delta}^{\tilde m^Y+\delta}\bar m^X(x)dx\right)^2,
\end{equation}
where $\tilde m^X$ and $\tilde m^Y$ are the mean of $\bar m^X:\int_{\mathbb{R}}xm^X(x)dx$ and of $\bar m^Y:\int_{\mathbb R}xm^Y(x)dx$, respectively.
Or
\begin{equation}
\label{eq:g_3}
\psi_3(\bar m^X,\bar m^Y)=(\tilde m^X-\tilde m^Y)^2,
\end{equation}
\noindent
Note that, since $\ell\equiv 0$, then the lower and upper Isaacs equations coincide and hence, under our hypotheses on controls as in the previous sections, the game has a value, which in the  sequel we are going to denote by $V$. Moreover, in all examples, besides the final datum, the Isaacs problem is given by
\begin{equation}
\label{eq:Isaacs_l0}
-V_t+\sup_a\langle D_XV,(a\bar m^X)_x\rangle+\inf_{b}\langle D_YV,(b\bar m^Y)_x\rangle=0.
\end{equation}
The cases \eqref{eq:g_1}, \eqref{eq:g_2} seem to be not immediate to be treated ``by hand'', as the example 1 in \cite{BCM}, because they are likely to involve a change in the shape of the masses making the solution highly dependent upon the Fr\'echet differentials. However, we are going to briefly argue about the case \eqref{eq:g_1} in the present {\it second example}. 
\par\smallskip
{\it First example}. The case \eqref{eq:g_3} is instead easier as example 2 in \cite{BCM}, giving a ``rigid optimal movement''. Indeed, suppose that both masses have mass equal $1$, compact support and that the admissible controls take values in the interval $[-c,c]$. Then the value function is the constant in time:
$$
V(\bar m^X,\bar m^Y,t)=\psi_3(\bar m^X,\bar m^Y)=(\tilde m^X-\tilde m^Y)^2.
$$
Indeed, $V_t=0, D_XV=2(\tilde m^X-\tilde m^Y)x,D_YV=-2(\tilde m^X-\tilde m^Y)x$, and hence \eqref{eq:Isaacs_l0} becomes
$$
-0+2c|\tilde m^X-\tilde m^Y|-2c|\tilde m^X-\tilde m^Y|=0,
$$
where, for example, we have used the fact that, integrating by parts and recalling that $m^X$ is positive, has compact support and total mass $1$,
$$
2(\tilde m^X-\tilde m^Y)\sup_a\int_{\mathbb{R}}x(a\bar m^X)_xdx=-2(\tilde m^X-\tilde m^Y)\sup_a\int_{\mathbb{R}}a\bar m^Xdx=2c|\tilde m^X-\tilde m^Y|.
$$
The optimal movement is then ``rigid'': the (Nash) equilibrium is given by simultaneously moving at the maximal velocity $c$ for the agents of both masses and in the same suitable one of the two directions (left or right).
\par\medskip
{\it Second example.} Here we consider the problem with zero running cost and final cost \eqref{eq:g_1}
$$
\psi_1(\bar m^X,\bar m^Y)=\int_{\mathbb{R}}\bar m^X(x)\bar m^Y(x)dx.
$$
We assume that the admissible controls at disposal of the two masses are constant with respect to $x\in\mathbb{R}$ (and hence the movements are rigid) and that the velocities are the same. This means that the sets of controls \eqref{campispaziali} are the same for both masses and that $|a|,|b|\le c$, $a_x,b_x=0$. In this case, the value function turns out to be  constant in time, that is
$$
V(\bar m^X,\bar m^Y,t)=\psi_1(\bar m^X,\bar m^Y)=\int_{\mathbb{R}}\bar m^X(x)\bar m^Y(x)dx.
$$
Indeed, using the fact that the controls are constant w.r.t. $x$, we have
\begin{equation}
\label{array1}
\begin{array}{ll}
\displaystyle
V_t=0,\ D_{X}V=\bar m^Y,\ D_{Y}V=\bar m^X,\\
\displaystyle
\sup_{a\in[-c,c]}\langle D_{X}V,(a\bar m^X)_x\rangle=\sup_{a\in[-c,c]}\left(a\int_{\mathbb{R}}\bar m^X_x\bar m^Ydx\right)=c\left|\int_{\mathbb{R}}\bar m^X_x\bar m^Ydx\right|,\\
\displaystyle
\inf_{b\in[-c,c]}\langle D_{Y}V,(b\bar m^Y)_x\rangle=\inf_{b\in[-c,c]}\left(b\int_{\mathbb{R}}\bar m^Y_x\bar m^Xdx\right)=-c\left|\int_{\mathbb{R}}\bar m^Y_x\bar m^Xdx\right|.
\end{array}
\end{equation}
But, since the masses have compact supports, it is
$$
\int_{\mathbb{R}}(\bar m^X\bar m^Y)_xdx=0\ \Rightarrow\ \int_{\mathbb{R}}\bar m^X_x\bar m^Ydx=-\int_{\mathbb{R}}\bar m^Y_x\bar m^Xdx,
$$
and we conclude that the sum of the last two lines in \eqref{array1} is, for some optimal $\tilde c\in\{-c,c\}$,
$$
\tilde c\left(\int \bar m^X_x\bar m^Ydx+\int \bar m^Y_x\bar m^Xdx\right)=\tilde c\int(\bar m^X\bar m^Y)_xdx=0,
$$
and then equation \eqref{eq:Isaacs_l0} is satisfied by $V$. Also in this case the (Nash) equilibrium is given by a simultaneous movement of both masses at the maximal velocity $c$ in the direction given by the sign of $\tilde c$.
\par
Of course, a more general and realistic model should require, as admissible, controls which are also space-dependent. In that case, the optimal behavior in general is not a rigid movement but involves a change masses' shape. For example, suppose that a population of lions in the savannah is initially mixed with a population of antelopes. That is the mass of the antelopes, $m^A$, has a support that contains in its interior the support of the mass of the lions, $m^L$. Assuming the whole mass of lions equal to $1$, at the initial time it is
$$
\int_{\mathbb{R}}m^Am^Ldx\ge a'>0,
$$
where $a'>0$ is such that $m^A\ge a'$ in the support of $m^L$. It is quite evident that a good behavior of the antelopes is to scatter as much as possible, with maximal velocity (which is $c>0$, the same of the lions by hypothesis) at the boundary of their support, obtaining then $m^A\le a''<a'$ in the support of the lions and hence
$$
\int_{\mathbb{R}}m^Am^Ldx\le a''<a',
$$
\noindent
showing that a change of shape of the mass $m^A$ is in this case necessary for optimization. A possible control for the antelopes at time $t$, in order to perform as above, is, if at the time $t$ the support of the antelopes is $[\xi_1,\xi_2]\subseteq\mathbb{R}$,
$$
a(x,t)=\begin{cases}
-c,&x\le \xi_1\\
c,&x\ge \xi_2\\
\text{linearly interpolating between $-c$ and $c$ in $[\xi_1,\xi_2]$.}
\end{cases}
$$
Of course, a simple guess with your hands for the optimal behavior of lions and antelopes, as function of $m^A,m^L$ and their shapes, is, in general, almost unachievable.

\section{Something on a possible parabolic case}
\label{sec:parabolic}
The model described in the previous sections is in some sense ``centralized'' in the evolution of the masses. The fields $\beta$ in \eqref{eq:intro_continuity} are not ``constructed'' by the single choices of the single agents of the masses. A natural question may be to weaken such a centralized feature. A first step could be to introduce a stochastic disturbance in the evolution of the single agents represented by the ordinary system \eqref{ode}. This is also a natural way to look at the movement of a huge number of agents/particles forming a mass. We then may assume that the evolution of the single agents of the masses are subject to a stochastic noise in their motion, that is, system \eqref{ode} is replaced by
$$
dY_s(s)=\beta(Y_s,s)+\sqrt{2\sigma}B_s,
$$
where $B_s$ is a standard Brownian motion on $\mathbb{R}^d$, and $\sigma>0$. In this case, the expected equation for the motion of the mass is the so-called Fokker-Planck parabolic equation
\begin{equation}
\label{eq:parabolic}
m_t-\sigma\Delta m+\mbox{div}(\beta m)=0.
\end{equation}
Whenever the field $\beta$ is sufficiently regular (and the hypotheses on the admissible controls for our problem (see \eqref{controlsX}-\eqref{controlsY}), together with the fact that they can be assumed with compact support, are certainly in that direction), and the initial datum is also regular, the solution of the Fokker-Planck equation is expected to belong to $C^0([0, T], H^1(\mathbb{R}^d))$ (see \cite{lady,car}), and one also expects that, for a test function $\varphi\in C^1(L^2(\R^d))$, the composition $\psi:t\longmapsto\varphi(m(t))$ (see Remark \ref{osscontinuita}) is differentiable and satisfies
 $$
\psi'(t)=\langle D_m\varphi(m),\sigma\Delta m-\mbox{div}(\beta m)\rangle_{L^2(\R^d)}
$$
or, for a test function $\varphi\in C^1(H^1(\R^d))$,
$$
\psi'(t)=-\langle \nabla D_m\varphi(m),\sigma\nabla m\rangle_{L^2(\R^d)}-\langle D_m\varphi(m),\mbox{div}(\beta m)\rangle_{L^2(\R^d)},
$$
where $\nabla D_m\varphi(m)\in L^2(\R^d)$ is the (spatial) gradient of  $D_m\varphi(m)$ when it can be identified with a function in $H^1(\R^d)$. Hence, the expected Isaacs equations (here only the upper equations) are, respectively:
\begin{multline}
\label{firstequationex}
-\underbar{V}_t+\inf_{b}\sup_{a}\left\{\langle D_{X}\underbar{V},-\sigma\Delta m^X+{\rm div}(am^X)\rangle_{L^2(\R^d)}+\langle D_{Y}\underbar{V},-\sigma\Delta m^Y+{\rm div}(bm^Y)\rangle_{L^2(\R^d)}\right.\\
\left.-\ell(m^X,m^Y,t,a,b)\right\}=0
\end{multline}
or
\begin{multline}
\label{secondequationex}
-\underbar{V}_t+\inf_{b}\sup_{a}\left\{\langle \nabla D_{X}\underbar{V},\sigma\nabla m^X\rangle_{L^2(\R^d)}+\langle D_X\underbar{V},\mbox{div}(am^X)\rangle_{L^2(\R^d)}\right.\\
\left.+\langle \nabla D_{Y}\underbar{V},\sigma\nabla m^Y\rangle_{L^2(\R^d)}+\langle D_Y\underbar{V},\div(bm^Y)\rangle_{L^2(\R^d)}-\ell(m^X,m^Y,t, a, b)\right\}=0.
\end{multline}
Note that, in any case, such equations are infinite dimensional equations of first order because they involve only the (first) Fr\'echet differential with respect to $m^X$ and $m^Y$.
\par
A rigorous analysis, derivation and justification of such Isaacs equation will be the subject of future works. Moreover, an interesting vanishing-viscosity-type observation seems to reasonably hold in our case, and hence to be certainly worth studying. Let, for any $\sigma>0$, $u_\sigma$ be a viscosity solution of, let us say, the second equation \eqref{secondequationex}. Observe that equations \eqref{firstequationex}-\eqref{secondequationex} are expected to hold in a similar set as $\tilde\cM$ \eqref{insiemeM}, which is compact in $H^1(\R^d)\times H^1(\R^d)\times[0,T]$. Since for $\sigma=0$ the solution of \eqref{eq:parabolic} (i.e. the solution of \eqref{eq:intro_continuity}) has already the good regularity, one may expect that such a compact set $\tilde\cM$ can be somehow taken independent on $\sigma$. Hence the following result may be expected but certainly worth studying.
\par
Suppose that $u_\sigma$, for $\sigma\to0$ uniformly converges on $\tilde\cM$ to a function $u$. Then $u$ is a viscosity solution of \eqref{HJI}. Indeed, let $P_0\in\tilde\cM$ be a point of local strict maximum for $u-\varphi$. Since $\tilde\cM$ is compact, let $P_\sigma$ be a point of maximum in $N\cap\tilde\cM$ for $u_\sigma-\varphi$, where $N$ is the closed neighborhood of $P_0$ in which it is a strict point of maximum. Then by uniform convergence and strict maximality, $P_\sigma\to P_0$ as $\sigma\to0$. In particular, for $\sigma$ small enough, $P_\sigma$ is also a point of local maximum for $u_\sigma-\varphi$. Since $u_\sigma$ is subsolution of the corresponding equation, by the convergence and the equi-boundedness of the points $P_\sigma,P_0$ in the corresponding norm involved in $\tilde\cM$, we get that $u$ is a subsolution of \eqref{HJI}.
\par
We point out again that such a ``vanishing-viscosity'' is not given by adding a second order term to the infinite dimensional Isaacs equation and letting it go to zero, but it is given by adding a second order term in the controlled evolution of the masses. These are still deterministic equations because they come from a stochastic perturbation of the first order ordinary equations which give the motion of the microscopic agents of the masses.
\par\bigskip\noindent
No datasets were generated or analysed during the current study
\par
\end{document}